\documentclass[12pt]{article}    % Specifies the document style.
\usepackage{amsmath}
\usepackage{amssymb}
\usepackage{epsfig}
\usepackage{textcomp}
\usepackage{multirow}
\usepackage{amsthm}
\usepackage{graphicx}
\usepackage{cite}
\usepackage{amsfonts}
\usepackage{mathrsfs}
\usepackage{color}
\newtheorem{Theorem}{Theorem}[section]

\newtheorem{Definition}[Theorem]{Definition}

\begin{document}           % End of preamble and beginning of text.
	 %%%%%%%%%%%%%%%%%%%%%%%%%%%%%%%%%%%%%%%%%%%%%%%%%%%%%%%%%%%%%%%%%%%%%%%%%%%
	\title{\Large New numerical approach for fractional differential equations\footnotemark}
	
	\author{\small Abdon Atangana\footnotemark \;\;and Kolade M. Owolabi\footnotemark \\
		%\small \footnotemark[2]  Department of Mathematical Sciences,  Federal University of Technology, PMB 704, Akure, \\
		%\small Ondo State, Nigeria\\
		\small \footnotemark[2]\;\footnotemark[3] Institute for Groundwater Studies, Faculty of Natural and Agricultural Sciences \\
		\small University of the Free State, Bloemfontein 9300, South
		Africa}
	\date{}
	\maketitle
	\def\thefootnote{\fnsymbol{footnote}}
	\setcounter{footnote}{0} \footnotetext[2]{E-mail addresses:
		abdonatangana@yahoo.fr (A. Atangana); mkowolax@yahoo.com (K.M. Owolabi)}
	\noindent
	
%%==================================================================
	\begin{abstract}
		\noindent
	 In the present case, we propose the correct version of the fractional Adams-Bashforth methods which take into account the nonlinearity of the kernels including the power law for the Riemann-Liouville type, the exponential decay law for the Caputo-Fabrizio case and the Mittag-Leffler law for the Atangana-Baleanu scenario.The Adams-Bashforth method for fractional differentiation suggested and are commonly use in the literature nowadays is not mathematically correct and the method was derived without taking into account the nonlinearity of the power law kernel. Unlike the proposed version found in the literature, our approximation, in all the cases, we are able to recover the standard case whenever the fractional power $\alpha=1$. 	
	\end{abstract}
	{\bf 2010 Mathematics Subject Classification}: {26A33, 34A34, 65M06}    \\
	\\
	\noindent
	{\bf Keywords:}  {Caputo derivative; Fractional differential equation.}

%%===================================================================
\section{Introduction}
Fractional calculus is known to be a generalization of the standard or integer-order calculus with a history of not less than over three centuries. It can be dated back to Leibniz's letter to L'Hospital, in which the meaning of the one-half order derivative was first introduced and discussed \cite{Mil93}. Although fractional calculus has such a long history, most of the research conducted still stay in the realm of theory, due to the lack of proper mathematical analysis methods and real applications.

Until the past decades, when many researchers pointed out that fractional derivative and fractional differential equation do have many applications in various fields. Differential problems with fractional derivative order have now become the most useful and powerful tools for describing nonlinear phenomena that are encountered in many application areas of biology, chemistry, ecology, engineering and various domains of applied sciences. A lot of mathematical models, such as in viscoelastic mechanics,  acoustic dissipation, boundary layer effects in duct,  biomedical engineering, power-law phenomena in fluid and complex network, allometric scaling laws in mathematical ecology and epidemiology, control theory, continuous time random walk, dielectric polarization,  porous media, quantitative finance, quantum evolution of complex systems, L\'evy statistics, fractional Brownian, fractional signal and image processing,  electrode-electrolyte polarization, electromagnetic waves, filters motion, phase-locked loops and non-local phenomena have justified to give a better description of the phenomenon under investigation than models with the integer order derivative \cite{Ata16, Gou16}.

Nowadays, there have been a lot studies on approximate methods for fractional differential equations. For instance, Dithelm
et al. and Li et al., have reported some results on numerical fractional ordinary differential equations \cite{Dit03, Li11}. When seeking an approximate solution  to the fractional order ordinary and partial differential equations, among many other choices that have been used are include the Adomian decomposition, homotopy perturbation and differential transform methods, for example, see \cite{Li09, Mom07}. A lot has been reported in the literature on various  fractional derivatives, ranging from the Riemann-Liouville to Atangana-Baleanu fractional derivative versions \cite{Cap15, Cap16, Ata16a}. Not only that, when numerically simulating such models, different numerical approximation techniques have been adopted in both space \cite{Owo16, Owo16a, Owo16b, Owo17, Owo17a, Owo17b} and time \cite{Bal16, Gom16, Gom17}.

The Adams-Bashforth has been recognized as a great and powerful numerical method able to provide a numerical solution closer to the exact solution. This method was developed with the classical differentiation using the fundamental theorem of calculus and taking the difference between two times including $t_{n+1}$ and $t_n$. This method was later extended to the concept of fractional differentiation with Caputo and Riemann-Liouville derivatives, however, the adaptation was not mathematically correct as the kernel of fractional integration is non-linear. In addition to this when the fractional order $\alpha = 1$ with this fractional version, we do not recover the classical Adams-Bashforth numerical scheme. In this paper, we will propose a new Adams-Bashforth for fractional differentiation with Caputo, Caputo-Fabrizio and Atangana-Baleanu derivatives, this version takes into account the nonlinearity of the kernels including the power law for Riemann-Liouville case, the exponential decay law for Caputo-Fabrizio case and Mittag-Leffler for Atangana-Baleanu case. Indeed when the fractional order turns to 1 one is expected to recover the classical Adams-Bashforth method.

\section{Preliminaries}
In this section, we recall some basic definitions and properties of fractional calculus theory which are useful in the next sections.

\begin{Definition}\emph{
	The Caputo fractional derivative of order $\alpha>0$ is defined by
\begin{equation}
_0^C\mathsf{D}_t^\alpha y(t)=\frac{1}{\Gamma(\alpha)}\int_{0}^{t}f(\lambda,y(\lambda))(t-\lambda)^{\alpha-1}d\lambda
\end{equation}   }
\end{Definition}

\begin{Definition}\emph{
	Let $y\in H^1(a,b),\;b>a$ and $\alpha\in[0,1]$. Then the new Caputo version of fractional derivative is defined as
\begin{equation}
_0^{CF}\mathsf{D}_t^\alpha y(t)=\frac{M(\alpha)}{1-\alpha}\int_{a}^{t}y'(\tau)\exp\left[-\frac{\alpha}{1-\alpha}(t-\tau)\right]d\tau
\end{equation}
where $M(\alpha)$ is a normalization function, such that $M(0)=M(1)=1$ \cite{Ata16c, Cap15, Cap16, Los15}. Nevertheless, if the function $y\notin H^1[a,b]$ then, the new derivative called the Caputo-Fabrizio fractional derivative can be defined as
\begin{equation}
_0^{CF}\mathsf{D}_t^\alpha y(t)=\frac{M(\alpha)}{1-\alpha}\int_{0}^{t}y'(\tau)\exp\left[-\frac{\alpha}{1-\alpha}(t-\tau)\right]d\tau.
\end{equation} }	
\end{Definition}

\begin{Theorem}
	Let $y(t)$ be a function for which the Caputo-Fabrizio exists, then, the Sumudu transform of the Caputo-Fabrizio fractional derivative of $y(t)$ is given as
	\begin{equation}
	ST\left(^{CF}_0\mathcal{D}_t^\alpha\right)(y(t))=M(\alpha)\frac{SF(y(t))-y(0)}{1-\alpha+\alpha u}
	\end{equation}
\end{Theorem}
\begin{proof}
	See Atangana \cite{Ata16} for detail.
\end{proof}

Atangana and Baleanu \cite{Ata16a} proposed the following derivatives.
\begin{Definition}\emph{
		Let $y\in H^1(a,b),\;a<b,\;\;\alpha\in[0,1]$ then, the definition of the Atangana and Baleanu fractional derivative in Caputo sense is given as \cite{Ata16a}
		\begin{equation}\label{ABC}
		^{ABC}_a\mathcal{D}_t^\alpha[y(t)]=\frac{M(\alpha)}{1-\alpha}\int_{a}^{t}y'(\tau)E_\alpha\left[-\alpha\frac{(t-\tau)^\alpha}{1-\alpha}\right]d\tau
		\end{equation}
		where $M(\alpha)$ has the same properties as in the case of the Caputo-Fabrizio fractional derivative.}
\end{Definition}

The above definition is considered to be useful to discuss real world problems, and it will also be a great advantage when applying the Laplace transform to solve some real life (physical) models with initial conditions. However, it should be noted that  we do not recover the original function when $\alpha=0$ except when at the origin the function vanishes. To avoid this kind of problem, the following definition is proposed.
\begin{Definition}\emph{
		Let $y\in H^1(a,b),\;a<b,\;\;\alpha\in[0,1]$ then, the definition of the Atangana-Baleanu fractional derivative in Riemann-Liouville sense becomes \cite{Ata16a}
		\begin{equation}\label{ABRL}
		^{ABR}_a\mathcal{D}_t^\alpha[y(t)]=\frac{M(\alpha)}{1-\alpha}\frac{d}{dt}\int_{a}^{t}y(\tau)E_\alpha\left[-\alpha\frac{(t-\tau)^\alpha}{1-\alpha}\right]d\tau
		\end{equation}	   }
\end{Definition}
clearly, both equations (\ref{ABC}) and (\ref{ABRL}) have a non-local kernel. We also obtain zero whenever the function in equation (\ref{ABC}) is constant.
%\begin{Definition}
%	The The Atangana-Baleanu fractional derivative of order $\alpha>0$ is defined by
%	\begin{equation}
%	_0^{ABC}\mathsf{D}_t^\alpha y(t)=\frac{1-\alpha}{ABC(\alpha)}f(t,y(t))+\frac{\alpha}{ABC(\alpha)\Gamma(\alpha)}\int_{0}^{t}(t-\tau)^{\alpha-1}f(\tau,y(\tau))d\tau.
%	\end{equation}
%\end{Definition}

%%==================================================================
\section{Numerical techniques for fractional differential equations}
The aim of this section is to introduce a new numerical approximation approach  based on the Caputo, Caputo-Fabrizio and Atangana-Baleanu fractional derivatives for the discretization of fractional differential equations. We also give the stability and convergence results for each of the derivatives.

\subsection{The Caputo fractional derivative }
We consider the following fractional differential equation
\begin{equation}\label{main1}
_0^C\mathsf{D}_t^\alpha y(t)=f(t,y(t))
\end{equation}
By applying the fundamental theorem of calculus on equation (\ref{main1}), we obtain
\begin{equation}
y(t)-y(0)=\frac{1}{\Gamma(\alpha)}\int_{0}^{t}f(\lambda,y(\lambda))(t-\lambda)^{\alpha-1}d\lambda,
\end{equation}
thus at $t=t_{n+1},\;n=0,1,2,\ldots$, we obtain
\begin{equation}\label{main3}
y(t_{n+1})-y(0)=\frac{1}{\Gamma(\alpha)}\int_{0}^{t_{n+1}}(t_{n+1}-t)^{\alpha-1}f(t,y(t))dt
\end{equation}
and
\begin{equation}\label{main4}
y(t_{n})-y(0)=\frac{1}{\Gamma(\alpha)}\int_{0}^{t_{n}}(t_n-t)^{\alpha-1}f(t,y(t))dt
\end{equation}
By subtracting (\ref{main4}) from (\ref{main3}), we get
\begin{eqnarray}
y(t_{n+1})=y(t_n)+\frac{1}{\Gamma(\alpha)}\int_{0}^{t_{n+1}}(t_{n+1}-t)^{\alpha-1} f(t,y(t))dt+\frac{1}{\Gamma(\alpha)}\int_{0}^{t_n}(t_n-t)^{\alpha-1}f(t,y(t))dt.\nonumber
\end{eqnarray}
This implies that
\begin{equation}
y(t_{n+1})=y(t_n)+A_{\alpha,1}+A_{\alpha,2}
\end{equation}
where
$$A_{\alpha,1}=\frac{1}{\Gamma(\alpha)}\int_{0}^{t_{n+1}}(t_{n+1}-t)^{\alpha-1} f(t,y(t))dt $$
and
$$A_{\alpha,2}=\frac{1}{\Gamma(\alpha)}\int_{0}^{t_n}(t_n-t)^{\alpha-1}f(t,y(t))dt,$$
the function $f(t,y(t))$ can be approximated using the Lagrange interpolation as
\begin{eqnarray}
P(t)\simeq \frac{t-t_{n-1}}{t_n-t_{n-1}}f(t_n,y_n)+\frac{t-t_{n}}{t_{n-1}-t_n}f(t_{n-1},y_{n-1})\nonumber\\
=\frac{f(t_n,y_n)}{h}(t-t_{n-1})-\frac{f(t_{n-1},y_{n-1})}{h}(t-t_n).
\end{eqnarray}
Thus,
\begin{equation}
A_{\alpha,1}=\frac{f(t_n,y_n)}{h\Gamma(\alpha)}\int_{0}^{t_{n+1}}(t_{n+1}-t)^{\alpha-1}(t-t_{n-1})dt-\frac{f(t_{n-1},y_{n-1})}{\Gamma(\alpha)h}\int_{0}^{t_{n+1}}(t_{n+1}-t)(t-t_n)dt
\end{equation}
\begin{eqnarray}
A_{\alpha,1}&=&\frac{f(t_n,y_n)}{h\Gamma(\alpha)}\int_{0}^{t_{n+1}}y^{\alpha-1}(t_{n+1-y-t_{n-1}})dy-\frac{f(t_{n-1},y_{n-1})}{\Gamma(\alpha)h}\int_{0}^{t_{n+1}}y^{\alpha-1}(t_{n+1}-y-t_{n})dy\nonumber\\
&=&\frac{f(t_n,y_n)}{h\Gamma(\alpha)}\left\{\frac{2ht^{\alpha}_{n+1}}{\alpha}-\frac{t^{\alpha+1}_{n+1}}{\alpha+1}\right\}-  \frac{f(t_{n-1},y_{n-1})}{h\Gamma(\alpha)}\left\{\frac{ht^{\alpha}_{n+1}}{\alpha}-\frac{t^{\alpha+1}_{n+1}}{\alpha+1}\right\}
\end{eqnarray}
Therefore,
\begin{equation}
A_{\alpha,1}=\frac{f(t_n,y_n)}{h\Gamma(\alpha)}\left\{\frac{2h}{\alpha}t_{n+1}^\alpha-\frac{t_{n+1}^{\alpha+1}}{\alpha+1}  \right\}-\frac{f(t_{n-1},y_{n-1})}{h\Gamma(\alpha)}\left\{\frac{h}{\alpha}t_{n+1}^\alpha-\frac{t_{n+1}^{\alpha+1}}{\alpha+1}  \right\}.
\end{equation}

Similarly, we obtain
\begin{eqnarray}
A_{\alpha,2}&=&\frac{f(t_n,y_n)}{h\Gamma(\alpha)}\int_{0}^{t_n}(t_n-t)^{\alpha-1}(t-t_{n-1})dt-\frac{(t_{n-1},y_{n-1})}{h\Gamma(\alpha)}\int_{0}^{t_n}(t_n-t)^{\alpha-1}(t-t_n)dt\nonumber\\
&=&\frac{f(t_n,y_n)}{h\Gamma(\alpha)}\int_{0}^{t_n}y^{\alpha-1}\left\{t_n-y-t_{n-1}\right\}dy + \frac{f(t_{n-1},y_{n-1})}{h\Gamma(\alpha)}\frac{t^{\alpha+1}_n}{\alpha}\\
&=&\frac{f(t_n,y_n)}{h\Gamma(\alpha)}\left\{\frac{ht^\alpha_{n}}{\alpha}-\frac{t^{\alpha+1}_n}{\alpha+1}\right\}+\frac{f(t_{n-1},y_{n-1})}{h\Gamma(\alpha+1)}t_n^{\alpha+1}.\nonumber
\end{eqnarray}
Thus the approximate solution is given as
\begin{eqnarray}
y(t_{n+1})&=&y(t_n)+\frac{f(t_n,y_n)}{h\Gamma(\alpha)}\left\{\frac{2h}{\alpha}t^\alpha_{n+1}-\frac{t_{n+1}^{\alpha+1}}{\alpha+1}+\frac{h}{\alpha}t_n^\alpha-\frac{t_n^{\alpha+1}}{\alpha} \right\}\nonumber\\
&&+\frac{f(t_{n-1},y_{n-1})}{h\Gamma(\alpha)}\left\{  \frac{h}{\alpha}t^\alpha_{n+1}-\frac{t_{n+1}^{\alpha+1}}{\alpha+1}+\frac{t_n^{\alpha}}{\alpha+1}  \right\}.
\end{eqnarray}

\begin{Theorem}
	Let
	$$_0^C\mathsf{D}_t^\alpha y(t)=f(t,y(t))$$
	be a fractional differential equation such that $f$ is bounded, then the numerical solution of $y(t)$ is given by
	\begin{eqnarray*}
	y(t_{n+1})&=&y(t_n)+\frac{f(t_n,y_n)}{h\Gamma(\alpha)}\left\{\frac{2h}{\alpha}t^\alpha_{n+1}-\frac{t_{n+1}^{\alpha+1}}{\alpha+1}+\frac{h}{\alpha}t_n^\alpha-\frac{t_n^{\alpha+1}}{\alpha} \right\}\nonumber\\
	&&+\frac{f(t_{n-1},y_{n-1})}{h\Gamma(\alpha)}\left\{  \frac{h}{\alpha}t^\alpha_{n+1}-\frac{t_{n+1}^{\alpha+1}}{\alpha+1}+\frac{t_n^{\alpha}}{\alpha+1}  \right\}+R_n^\alpha(t)
	\end{eqnarray*}
	where
	$$R_n^\alpha(t)<\frac{h^{3+\alpha}M}{12\Gamma(\alpha+1)}\left\{(n+1)^\alpha+n^2\right\}$$
\end{Theorem}
\begin{proof}
	Following the derivation presented earlier, we have
	\begin{equation*}
	y(t_{n+1})=y(t_n)+\frac{1}{\Gamma(\alpha)}\int_{0}^{t_{n+1}}(t_{n+1}-t)^{\alpha-1}f(t,y(t))dt-\frac{1}{\Gamma(\alpha)}\int_{0}^{t_n}(t_{n+1}-t)^{\alpha-1} f(t,y(t))dt.
	\end{equation*}
	Using the Lagrange polynomial, we have
	\begin{eqnarray}
	y(t_{n+1})-y(t_n)&=&\frac{1}{\Gamma(\alpha)}\int_{0}^{t_{n+1}}L_1(t)(t_{n+1}-t)^{\alpha-1}dt + \frac{1}{\Gamma(\alpha)}\int_{0}^{t_{n+1}}R_1(t)(t_{n+1}-t)^{\alpha-1}dt\nonumber\\
	&&-\frac{1}{\Gamma(\alpha)}\int_{0}^{t_{n}}L_2(t)(t_{n}-t)^{\alpha-1}dt + \frac{1}{\Gamma(\alpha)}\int_{0}^{t_{n}}R_2(t)(t_{n}-t)^{\alpha-1}dt\nonumber\\
	&=& \frac{f(t_n,y_n)}{h\Gamma(\alpha)}\left\{\frac{2h}{\alpha}t^\alpha_{n+1}-\frac{t_{n+1}^{\alpha+1}}{\alpha+1}+\frac{ht^\alpha}{\alpha}-\frac{t^{\alpha+1}}{\alpha+1}\right\}+\frac{f(t_{n-1},y_{n-1})}{h\Gamma(\alpha)}\nonumber\\
	&&\times \left\{\frac{ht_{n+1}^\alpha}{\alpha}-\frac{t_{n+1}^{\alpha+1}}{\alpha+1}+\frac{t_n^\alpha}{\alpha+1}\right\}+\frac{1}{\Gamma(\alpha)}\int_{0}^{t_{n+1}}R(t)(t_{n+1}-t)^{\alpha-1}dt\nonumber\\
	&& - \frac{1}{\Gamma(\alpha)}\int_{0}^{t_{n}}R(t)(t_{n}-t)^{\alpha-1}dt.
	\end{eqnarray}
	
	Next, we let
	$$R_n(t)=\frac{1}{\Gamma(\alpha)}\int_{0}^{t_{n+1}}R_1(t)(t_{n+1}-t)^{\alpha-1}dt-\frac{1}{\Gamma(\alpha)}\int_{0}^{t_{n}}R_1(t)(t_{n}-t)^{\alpha-1}dt$$
	where
	$$R_1(t)=\frac{f^{(n+1)}(t,y(t))}{(n+1)!}\prod_{i=0}^{n}(t-t_i)$$
	\begin{eqnarray*}
	R_n(t)&=&\frac{1}{\Gamma(\alpha)}\int_{0}^{t_{n+1}}\frac{f^{(n+1)}(t,y(t))}{(n+1)!}\prod_{i=0}^{n}(t-t_i)(t_{n+1}-t)^{\alpha-1}dt\\
	&& -
	\frac{1}{\Gamma(\alpha)}\int_{0}^{t_{n}}\frac{f^{(n+1)}(t,y(t))}{(n+1)!}\prod_{i=0}^{n}(t-t_i)(t_{n}-t)^{\alpha-1}dt.
	\end{eqnarray*}
	
	Let $\|f(t)\|_\infty=\sup_{t\in[a,b]}|f(t)|$, we assume that
	$$\|f^{(3)}(t,y(t))\|_\infty<M<\infty$$
	then
	\begin{eqnarray*}
		\|R_n(t)\|_\infty&=&\left\|\frac{1}{\Gamma(\alpha)}\int_{0}^{t_{n+1}}\frac{f^{(3)}(t,y(t))}{(n+1)!}\prod_{i=0}^{n}(t-t_i)(t_{n+1}-t)^{\alpha-1}dt\right.\\
		&&\left. -
		\frac{1}{\Gamma(\alpha)}\int_{0}^{t_{n}}\frac{f^{(n+1)}(t,y(t))}{(n+1)!}\prod_{i=0}^{n}(t-t_i)(t_{n}-t)^{\alpha-1}dt\right\|_\infty
	\end{eqnarray*}
	
	\begin{eqnarray*}
		\|R_n(t)\|_\infty&\le&\left\|\frac{1}{\Gamma(\alpha)}\int_{0}^{t_{n+1}}\frac{f^{(n+1)}(t,y(t))}{(n+1)!}\prod_{i=0}^{n}(t-t_i)(t_{n+1}-t)^{\alpha-1}dt\right.\\
		&&\left. -
		\frac{1}{\Gamma(\alpha)}\int_{0}^{t_{n}}\frac{f^{(n+1)}(t,y(t))}{(n+1)!}\prod_{i=0}^{n}(t-t_i)(t_{n}-t)^{\alpha-1}dt\right\|_\infty.
	\end{eqnarray*}
Without loss of generality, we evaluate
\begin{eqnarray*}
&&\left\|\frac{1}{6\Gamma(\alpha)}\int_{0}^{t_{n+1}}f^{(3)}(t,y(t)) \prod_{i=0}^{2}(t-t_i)(t_n-t)^{\alpha-1}dt   \right\|_\infty\\
&&\hspace{2cm}\le \frac{1}{6\Gamma(\alpha)}\int_{0}^{t_{n+1}}\left\|f^{(3)}(t,y(t)) \right\|_\infty \left\|\prod_{i=0}^{2}(t-t_i) \right\| (t_{n+1}-t)^{\alpha-1}dt \\
&&\hspace{2cm}\le \frac{1}{6\Gamma(\alpha)}\left\|f^{(3)}(t,y(t)) \right\|_\infty \left\|\prod_{i=0}^{2}(t-t_i) \right\| \int_{0}^{t_{n+1}}(t_{n+1}-t)^{\alpha-1}dt \\
&&\hspace{2cm}\le\frac{2h^3}{4\times 6\Gamma(\alpha)} \left\|f^{(3)}(t,y(t)) \right\|_\infty \frac{t_{n+1}^\alpha}{\alpha}\\
&&\hspace{2cm}<\frac{h^3t_{n+1}^\alpha M}{12\Gamma(\alpha+1)},
\end{eqnarray*}
therefore,
$$\left\|\frac{1}{\Gamma(\alpha)}\int_{0}^{t_{n+1}}\frac{f^{(n+1)}(t,y(t))}{(n+1)!} \prod_{i=0}^{n}(t-t_i)(t_{n+1}-t)^{\alpha-1}dt \right\|_\infty<\frac{h^3t_{n+1}^\alpha M}{12\Gamma(\alpha+1)}$$
and
$$\left\|\frac{1}{\Gamma(\alpha)}\int_{0}^{t_{n}}\frac{f^{(n+1)}(t,y(t))}{(n+1)!} \prod_{i=0}^{n}(t-t_i)(t_{n}-t)^{\alpha-1}dt \right\|_\infty<\frac{h^3t_{n}^\alpha M}{12\Gamma(\alpha+1)}$$	
thus
\begin{eqnarray*}
\|R(t)\|_\infty&<&\frac{h^3M}{12\Gamma(\alpha+1)}(t^\alpha_{n+1}+t^\alpha_n)\\
&<&\frac{h^{3+\alpha}M}{12\Gamma(\alpha+1)}\left((n+1)^\alpha+n^\alpha\right)
\end{eqnarray*}
Again, we evaluate
\begin{eqnarray*}
\|y_{n+1}-y_n\|_\infty&=&\left\|\frac{1}{\Gamma(\alpha)}\int_{0}^{t_{n+1}}(t_{n+1}-t)^{\alpha-1}f(t,y(t))dt\right.\\
&& \left.- \frac{1}{\Gamma(\alpha)}\int_{0}^{t_{n}}(t_{n}-t)^{\alpha-1}f(t,y(t))dt \right\|_\infty\\
&\le& \left\|\frac{1}{\Gamma(\alpha)}\int_{0}^{t_{n+1}}(t_{n+1}-t)^{\alpha-1}f(t,y(t))dt\right\|_\infty+\\
&& \left\| \frac{1}{\Gamma(\alpha)}\int_{0}^{t_{n}}(t_{n}-t)^{\alpha-1}f(t,y(t))dt \right\|_\infty.
\end{eqnarray*}
Without loss of generality, we focus on
\begin{eqnarray*}
&&\left\|\frac{1}{\Gamma(\alpha)}\int_{0}^{t_{n+1}}(t_{n+1}-t)^{\alpha-1}f(t,y(t))dt\right\|_\infty\\
&&\hspace{2cm}\le\left\|\frac{1}{\Gamma(\alpha)}\sum_{j=0}^{n}\int_{t_j}^{t_{j+1}}(t_{n+1}-t)^{\alpha-1} f(t_j,y(t_j))dt\right\|_\infty\\
&&\hspace{2cm}\le \frac{1}{\Gamma(\alpha)}\sum_{j=0}^{n}\int_{t_j}^{t_{j+1}}(t_{n+1}-t)^{\alpha-1} \left\| f(t_j,y(t_j))dt\right\|_\infty\\
&&\hspace{2cm}\le \frac{1}{\Gamma(\alpha)} \max_{t\in[t_j,t_{j+1}]} \|f(t_j,y(t_j))\|\sum_{j=0}^{n} \int_{t_j}^{t_{j+1}}(t_{n+1}-t)^{\alpha-1}dt\\
&&\hspace{2cm}\le \frac{1}{\Gamma(\alpha)} \max_{t\in[t_j,t_{j+1}]} \|f(t_j,y(t_j))\|\sum_{j=0}^{n}\left\{\frac{(t_{n+1}-t_j)^\alpha}{\alpha}-\frac{(t_{n+1}-t_{j+1})^\alpha}{\alpha}\right\}\\
&&\hspace{2cm}\le \frac{h^\alpha}{\Gamma(\alpha+1)} \max_{t\in[t_j,t_{j+1}]} \|f(t_j,y(t_j))\|\sum_{j=0}^{n}\left\{(n+1-j)^\alpha-(n-j)^\alpha \right\}\\
&&\hspace{2cm}\le \frac{h^\alpha}{\Gamma(\alpha+1)} \max_{t\in[t_j,t_{j+1}]} \|f(t_j,y(t_j))\|\left\{(n+1)^\alpha+1\right\}\\
&&\hspace{2cm}\le \frac{h^\alpha}{\Gamma(\alpha+1)} \max_{t\in[t_j,t_{j+1}]} \|f(t,y(t))\|\left\{(n+1)^\alpha+1\right\}
\end{eqnarray*}
Thus
$$\|\frac{1}{\Gamma(\alpha)}\int_{0}^{t_{n+1}}(t_{n+1}-t)^{\alpha-1}f(t,y(t))\|<\frac{h^\alpha C_n^\alpha}{\Gamma(\alpha+1)}\max_{t\in[t_j,t_{j+1}]} \|f(t,y(t))\|$$ and
$$\|\frac{1}{\Gamma(\alpha)}\int_{0}^{t_{n}}(t_{n}-t)^{\alpha-1}f(t,y(t))\|<\frac{h^\alpha S_n^\alpha}{\Gamma(\alpha+1)}\max_{t\in[t_j,t_{j+1}]} \|f(t,y(t))\|$$
Therefore
$$\|y_{n+1}-y_n\|_\infty <\frac{h^\alpha}{\Gamma(\alpha+1)}\max_{t\in[t_j,t_{j+1}]} \|f(t,y(t))\|\left\{C_n^\alpha +S_n^\alpha \right\}.$$
The proof is completed.
\end{proof}

\subsection{The Caputo-Fabrizio fractional derivative }
We next consider the following general fractional differential equation with fading memory included via the Caputo-Fabrizio fractional derivative. That is,
\begin{equation}\label{cf1}
_0^{CF}\mathsf{D}_t^\alpha y(t)=f(t,y(t))
\end{equation}
or
\begin{equation}\label{cf2}
\frac{M(\alpha)}{1-\alpha}\int_{0}^{t}y'(\tau)\exp\left[-\frac{\alpha}{1-\alpha}(t-\tau)\right]d\tau=f(t,y(t)).
\end{equation}
Using the fundamental theorem of calculus, we convert the above to
\begin{equation}\label{cf3}
y(t)-y(0)=\frac{1-\alpha}{M(\alpha)}f(t,y(t))+\frac{\alpha}{M(\alpha)}\int_{0}^{t}f(\tau,y(\tau))d\tau
\end{equation}
so that
\begin{equation}\label{cf4}
y(t_{n+1})-y(0)=\frac{1-\alpha}{M(\alpha)}f(t_n,y(t_n))+\frac{\alpha}{M(\alpha)}\int_{0}^{t_{n+1}}f(t,y(t))dt
\end{equation}
and
\begin{equation}\label{cf5}
y(t_{n})-y(0)=\frac{1-\alpha}{M(\alpha)}f(t_{n-1},y(t_{n-1}))+\frac{\alpha}{M(\alpha)}\int_{0}^{t_{n}}f(t,y(t))dt
\end{equation}
On removing (\ref{cf5}) from (\ref{cf4}) we obtain
\begin{equation}\label{cf6}
y(t_{n+1})-y(t_n)=\frac{1-\alpha}{M(\alpha)}\left\{f(t_n,y_n)-f(t_{n-1},y_{n-1})\right\}+\frac{\alpha}{M(\alpha)}\int_{t_n}^{t_{n+1}}f(t,y(t))dt
\end{equation}
where
\begin{eqnarray}
\int_{t_n}^{t_{n+1}}f(t,y(t))dt&=&\int_{t_n}^{t_{n+1}}\left\{\frac{f(t_n,y_n)}{h}(t-t_{n-1})-\frac{f(t_{n-1},y_{n-1})}{h}(t-t_{n})\right\}dt\nonumber\\
&=&\frac{3h}{2}{f(t_n,y_n)}-\frac{h}{2}f(t_{n-1},y_{n-1}).
\end{eqnarray}
Thus,
\begin{eqnarray*}
y(t_{n+1})-y(t_n)&=&\frac{1-\alpha}{M(\alpha)}=\left[f(t_n,y_n)-f(t_{n-1},y_{n-1})\right]+ \frac{3\alpha h}{2M(\alpha)}{f(t_n,y_n)}\\
&&-\frac{\alpha h}{2M(\alpha)}f(t_{n-1},y_{n-1})
\end{eqnarray*}
which implies that
\begin{equation*}
y(t_{n+1})-y(t_n)=\left(\frac{1-\alpha}{M(\alpha)}+\frac{3\alpha h}{2M(\alpha)}\right)f(t_n,y_n)+ \left(\frac{1-\alpha}{M(\alpha)}+\frac{\alpha h}{2M(\alpha)}\right)f(t_{n-1},y_{n-1}).
\end{equation*}
Hence,
\begin{equation}
y_{n+1}=y_n+\left(\frac{1-\alpha}{M(\alpha)}+\frac{3\alpha h}{2M(\alpha)}\right)f(t_n,y_n)+ \left(\frac{1-\alpha}{M(\alpha)}+\frac{\alpha h}{2M(\alpha)}\right)f(t_{n-1},y_{n-1})
\end{equation}
which is the corresponding two-step Adams-Bashforth method for the Caputo-Fabrizio fractional derivative. The following theorems present the convergence and stability results.

% % % % % % % %
\begin{Theorem}
	Let $y(t)$ be a solution of $_0^{CF}\mathsf{D}_t^\alpha y(t)=f(t,y(t))$ where $f$ is a continuous function bounded for the Caputo-Fabrizio fractional derivative, we have
\begin{equation*}
y_{n+1}=y_n+\left(\frac{1-\alpha}{M(\alpha)}+\frac{3\alpha h}{2M(\alpha)}\right)f(t_n,y_n)+ \left(\frac{1-\alpha}{M(\alpha)}+\frac{\alpha h}{2M(\alpha)}\right)f(t_{n-1},y_{n-1}) + R^n_\alpha
\end{equation*}
where $\|R^n_\alpha\|\le M$.
\end{Theorem}

\begin{proof}
According to the definition of the Caputo-Fabrizio derivative, we have
\begin{equation}
y(t)-y(0)=\frac{1-\alpha}{M(\alpha)}f(t,y(t))+\frac{\alpha}{M(\alpha)}\int_{0}^{t}f(\tau,y(\tau))d\tau.
\end{equation} 	
At point $t_{n+1}$, we get
$$y(t_{n+1})= \frac{1-\alpha}{M(\alpha)}f(t_n,y(t_n))+\frac{\alpha}{M(\alpha)}\int_{0}^{t_{n+1}}f(t,y(t))dt,$$
also, at point $t_n$ we have
$$y(t_{n})= \frac{1-\alpha}{M(\alpha)}f(t_{n-1},y(t_{n-1}))+\frac{\alpha}{M(\alpha)}\int_{0}^{t_{n}}f(t,y(t))dt.$$
Thus,
\begin{eqnarray}
y(t_{n+1})-y(t_n)&=&\frac{1-\alpha}{M(\alpha)}\left\{f(t_n,y_n)-f(t_{n-1},y(t_{n-1})) \right\} +\frac{\alpha}{M(\alpha)}\int_{t_n}^{t_{n+1}}f(t,y(t))dt\nonumber\\
&=& \frac{1-\alpha}{M(\alpha)}\left\{f(t_n,y_n)-f(t_{n-1},y(t_{n-1})) \right\} + \frac{\alpha}{M(\alpha)}\left\{\frac{f(t_n,y_n)}{h} (t-t_{n-1})\right.\nonumber\\
&&\left.-\frac{f(t_{n-1},y_{n-1})}{h}(t-t_n)+\sum_{i=2}^{n}\prod_{i=2}^{n}\frac{(t-t_i)}{h(-1)^i}f(t_i,y_i) \right\}dt
\end{eqnarray}
so that
\begin{eqnarray}
y_{n+1}&=&y_n+\left(\frac{1-\alpha}{M(\alpha)}+\frac{3\alpha h}{2M(\alpha)}\right)f(t_n,y_n)+ \left(\frac{1-\alpha}{M(\alpha)}+\frac{\alpha h}{M(\alpha)}\right)f(t_{n-1},y_{n-1})\nonumber\\
&&+\underbrace{\frac{\alpha}{M(\alpha)}\int_{t_n}^{t_{n+1}}\sum_{i=2}^{n}\prod_{i=2}^{n}\frac{(t-t_i)}{h(-1)^i} f(t_i,y_i)dt}_{\text{error term}}
\end{eqnarray}

We denote the above error term by
$$R_\alpha^n=\frac{\alpha}{M(\alpha)}\int_{t_n}^{t_{n+1}}\sum_{i=2}^{n}\prod_{i=2}^{n}\frac{(t-t_i)}{h(-1)^i} f(t_i,y_i)dt$$
Thus
\begin{eqnarray}
\left\| R_\alpha^n \right\|_\infty&=& \frac{\alpha}{M(\alpha)}\left\|\int_{t_n}^{t_{n+1}}\sum_{i=2}^{n}\prod_{i=2}^{n}\frac{(t-t_i)}{h(-1)^i} f(t_i,y_i)dt \right\|_\infty\nonumber\\
&\le& \frac{\alpha}{M(\alpha)}\int_{t_n}^{t_{n+1}}\left\|\sum_{i=2}^{n}\prod_{i=2}^{n}\frac{(t-t_i)}{h(-1)^i} f(t_i,y_i)dt \right\|_\infty\nonumber\\
&\le& \frac{\alpha}{M(\alpha)} \int_{t_n}^{t_{n+1}}\sum_{i=2}^{n}\prod_{i=2}^{n} \left|\frac{(t-t_i)}{h(-1)^i}\right| \left\|f(t_i,y_i)dt \right\|_\infty \nonumber\\
&<& \frac{\alpha}{M(\alpha)} \int_{t_n}^{t_{n+1}}\sum_{i=2}^{n}\prod_{i=2}^{n} \frac{\left|(t-t_i)\right|}{h}\sup_{i\in I}\left\{\max_{i\in I} |f(t_i,y_i)|\right\}\nonumber\\
&<&\frac{\alpha}{M(\alpha)}\frac{(n+1)!h^{n+1}}{4}M.
\end{eqnarray}
Hence,
$$\left\| R_\alpha^n \right\|_\infty <\frac{\alpha}{M(\alpha)}(n+1)!h^{n+1}M.$$
\end{proof}

\begin{Theorem}
	Let $y(t)$ be a solution of $_0^{CF}\mathsf{D}_t^\alpha y(t)=f(t,y(t))$, for every $n\in N$
	$$\|y_{n+1}-y_n\|_\infty < \frac{1-\alpha}{M(\alpha)}\left\|f(t_n,y_n)-f(t_{n-1},y_{n-1}) \right\|_\infty +\frac{\alpha h^{n+1}(n+1)!}{4M(\alpha)}$$
	such that if $\|f(t_n,y_n)-f(t_{n-1},y_{n-1})\|_\infty \rightarrow 0$ as $n\rightarrow \infty$, then $\|y_{n+1}-y_n\|_\infty \rightarrow 0$ as $n\rightarrow \infty$.
\end{Theorem}

\begin{proof}
	\begin{eqnarray}
	\|y_{n+1}-y_n\|_\infty&=& \left\|\frac{1-\alpha}{M(\alpha)}\left\{f(t_n,y_n)-f(t_{n-1},y_{n-1}) \right\}+\frac{\alpha}{M(\alpha)}\int_{t_n}^{t_{n+1}}f(t,y(t))dt \right\|_\infty\nonumber\\
	&\le& \frac{1-\alpha}{M(\alpha)}\left\|f(t_n,y_n)-f(t_{n-1},y_{n-1}) \right\|_\infty+\frac{\alpha}{M(\alpha)}\left\|\int_{t_n}^{t_{n+1}}f(t,y(t))dt \right\|_\infty\nonumber\\
	&\le& \frac{1-\alpha}{M(\alpha)}\left\|f(t_n,y_n)-f(t_{n-1},y_{n-1}) \right\|_\infty+\frac{\alpha}{M(\alpha)}\int_{t_n}^{t_{n+1}}\left|\sum_{i=0}^{n}\prod_{i=0}^{n}\frac{(t-t_i)}{(-1)h} \right|_\infty dt\nonumber\\
	&<& \frac{1-\alpha}{M(\alpha)}\left\|f(t_n,y_n)-f(t_{n-1},y_{n-1}) \right\|_\infty +\frac{\alpha h^{n+1}(n+1)!}{4M(\alpha)}.
	\end{eqnarray}
	So if $\|f(t_n,y_n)-f(t_{n-1},y_{n-1})\|_\infty \rightarrow 0$ as $n\rightarrow \infty$, then $\|y_{n+1}-y_n\|_\infty \rightarrow 0$ as $n\rightarrow \infty$. This completes the proof.
\end{proof}
% % % %% % % %

\subsection{Atangana-Baleanu fractional derivative in Caputo sense}
Let us consider the following fractional differential equation
\begin{equation}\label{ab1}
_0^{ABC}\mathsf{D}_t^\alpha y(t)=f(t,y(t)).
\end{equation}
Again, we apply the fundamental theorem of calculus to have
\begin{equation}\label{ab2}
y(t)-y(0)=\frac{1-\alpha}{ABC(\alpha)}f(t,y(t))+\frac{\alpha}{ABC(\alpha)\Gamma(\alpha)}\int_{0}^{t}(t-\tau)^{\alpha-1}f(\tau,y(\tau))d\tau.
\end{equation}
At $t_{n+1}$, we have
$$y(t_{n+1})-y(0)=\frac{1-\alpha}{ABC(\alpha)}f(t_n,y_n)+\frac{\alpha}{ABC(\alpha)\Gamma(\alpha)}\int_{0}^{t_{n+1}}(t_{n+1}-\tau)^{\alpha-1}f(t,y(t))dt $$
and at $t_n$ we have
$$y(t_{n})-y(0)=\frac{1-\alpha}{ABC(\alpha)}f(t_{n-1},y_{n-1})+\frac{\alpha}{ABC(\alpha)\Gamma(\alpha)}\int_{0}^{t_{n}}(t_{n}-\tau)^{\alpha-1}f(t,y(t))dt $$
which on subtraction yields
\begin{eqnarray}
y(t_{n+1})-y(t_n)&=&\frac{1-\alpha}{ABC(\alpha)}\left\{f(t_n,y_n)-f(t_{n-1},y_{n-1})\right\}+\frac{\alpha}{ABC(\alpha)\Gamma(\alpha)}\int_{0}^{t_{n+1}}(t_{n+1}-t)^{\alpha-1}f(t,y(t))dt \nonumber\\
&&-  \frac{\alpha}{ABC(\alpha)\Gamma(\alpha)}\int_{0}^{t_{n}}(t_{n}-t)^{\alpha-1}f(t,y(t))dt.
\end{eqnarray}
Therefore,
$$y(t_{n+1})-y(t_n)=\frac{1-\alpha}{ABC(\alpha)}\left\{f(t_n,y_n)-f(t_{n-1},y_{n-1})\right\}+A_{\alpha,1}-A_{\alpha,2}.$$

Without loss of generality, we consider
$$A_{\alpha,1}=\frac{\alpha}{ABC(\alpha)\Gamma(\alpha)}\int_{0}^{t_{n+1}}(t_{n+1}-t)^{\alpha-1}f(t,y(t))dt$$
Again we consider the approximation
\begin{equation}
p(t)=\frac{t-t_{n-1}}{t_n-t_{n-1}}f(t_n,y_n)+ \frac{t-t_{n-1}}{t_{n-1}-t_{n}}f(t_{n-1},y_{n-1})
\end{equation}
thus
\begin{eqnarray}
A_{\alpha,1}&=&\frac{\alpha}{ABC(\alpha)\Gamma(\alpha)}\int_{0}^{t_{n+1}}(t_{n+1}-t)^{\alpha-1}\left\{\frac{t-t_{n-1}}{h}f(t_n,y_n)-\frac{t-t_n}{h}f(t_n,y_n)  \right\}\nonumber\\
&=&\frac{\alpha f(t_n,y_n)}{ABC(\alpha)\Gamma(\alpha)h}\left\{\int_{0}^{t_{n+1}}(t_{n+1}-t)^{\alpha-1}f(t-t_{n-1}) \right\}dt\nonumber\\
&& - \frac{\alpha f(t_{n-1},y_{n-1})}{ABC(\alpha)\Gamma(\alpha)h}\left\{\int_{0}^{t_{n+1}}(t_{n+1}-t)^{\alpha-1}f(t-t_{n-1}) \right\}dt\nonumber\\
&=&\frac{\alpha f(t_n,y_n)}{ABC(\alpha)\Gamma(\alpha)h}\left\{\frac{2h t_{n+1}^\alpha}{\alpha}-\frac{ t^{\alpha+1}_{n+1}}{\alpha+1} \right\} - \frac{\alpha f(t_{n-1},y_{n-1})}{ABC(\alpha)\Gamma(\alpha)h}\left\{\frac{h t_{n+1}^\alpha}{\alpha}-\frac{ t^{\alpha+1}_{n+1}}{\alpha+1} \right\}.
\end{eqnarray}
Similarly, we obtain
\begin{equation}
A_{\alpha,2}=\frac{\alpha f(t_n,y_n)}{ABC(\alpha)\Gamma(\alpha)h} \left\{\frac{h t_{n}^\alpha}{\alpha}-\frac{ t^{\alpha+1}_{n}}{\alpha+1} \right\}-\frac{ f(t_{n-1},y_{n-1})}{ABC(\alpha)\Gamma(\alpha)h}
\end{equation}
thus
\begin{eqnarray}
y(t_{n+1})-y(t_n)&=&\frac{1-\alpha}{ABC(\alpha)} \left\{f(t_n,y_n)-f(t_{n-1},y_{n-1})\right\} +\frac{\alpha f(t_n,y_n)}{ABC(\alpha)\Gamma(\alpha)h}\left\{\frac{2h t_{n+1}^\alpha}{\alpha}-\frac{t^{\alpha+1}_{n+1}}{\alpha+1}\right\}\nonumber\\
&&-\frac{\alpha f(t_{n-1},y_{n-1})}{ABC(\alpha)\Gamma(\alpha)h}\left\{\frac{h t_{n+1}^\alpha}{\alpha}-\frac{t^{\alpha+1}_{n+1}}{\alpha+1}\right\}-
\frac{\alpha f(t_n,y_n)}{ABC(\alpha)\Gamma(\alpha)h}\left\{\frac{h t_{n}^\alpha}{\alpha}-\frac{t^{\alpha+1}_{n}}{\alpha+1}\right\}\nonumber\\
&&+\frac{ f(t_{n-1},y_{n-1})}{ABC(\alpha)\Gamma(\alpha)}t_n^{\alpha+1}.
\end{eqnarray}

\begin{eqnarray}
y_{n+1}&=&y_n + f(t_n,y_n)\left\{\frac{1-\alpha}{ABC(\alpha)}+\frac{\alpha}{ABC(\alpha)h}\left[\frac{2h t_{n+1}^\alpha}{\alpha}-\frac{t^{\alpha+1}_{n+1}}{\alpha+1}  \right]\right.\nonumber\\
&&\left.-\frac{\alpha}{ABC(\alpha)\Gamma(\alpha)h}\left[\frac{ht^\alpha_n}{\alpha}-\frac{t^{\alpha+1}_n}{\alpha+1}\right]\right\} +f(t_{n-1},y_{n-1}) \\
&&\times\left\{\frac{\alpha-1}{ABC(\alpha)}-\frac{\alpha}{h\Gamma(\alpha)ABC(\alpha)}\left[\frac{ht_{n+1}^\alpha}{\alpha}-\frac{t_{n+1}^{\alpha+1}}{\alpha+1}+ \frac{t^{\alpha+1}}{h\Gamma(\alpha)ABC(\alpha)} \right]  \right\}.\nonumber
\end{eqnarray}
The above equation is called two-step Adams-Bashforth scheme for Atangana-Baleanu fractional derivative in the sense of Caputo. In what follows, we give the convergence and stability results.

\begin{Theorem}{\bf(Convergence result)}
	Let $y(t)$ be a solution of
	$$^{ABC}_0\mathsf{D}_t^\alpha y(t)=f(t,y(t))$$
	with $f$ being continuous and bounded, the numerical solution of y(t) is given as
\begin{eqnarray*}
y_{n+1}&=&y_n + f(t_n,y_n)\left\{\frac{1-\alpha}{ABC(\alpha)}+\frac{\alpha}{ABC(\alpha)h}\left[\frac{2h t_{n+1}^\alpha}{\alpha}-\frac{t^{\alpha+1}_{n+1}}{\alpha+1}  \right]\right.\nonumber\\
&&\left.-\frac{\alpha}{ABC(\alpha)\Gamma(\alpha)h}\left[\frac{ht^\alpha_n}{\alpha}-\frac{t^{\alpha+1}_n}{\alpha+1}\right]\right\} +f(t_{n-1},y_{n-1}) \\
&&\times\left\{\frac{\alpha-1}{ABC(\alpha)}-\frac{\alpha}{h\Gamma(\alpha)ABC(\alpha)}\left[\frac{ht_{n+1}^\alpha}{\alpha}-\frac{t_{n+1}^{\alpha+1}}{\alpha+1}+ \frac{t^{\alpha+1}}{h\Gamma(\alpha)ABC(\alpha)} \right]  \right\}+R_\alpha
\end{eqnarray*}
where $\|R_\alpha\|_\infty<M$
\end{Theorem}

\begin{proof}
\begin{eqnarray}
y_{n+1}-y_n&=&\frac{1-\alpha}{ABC(\alpha)}\left(f(t_n,y_n)-f(t_{n-1},y_{n-1})\right)+\frac{\alpha}{ABC(\alpha)\Gamma(\alpha)}\times\nonumber\\
&&\left[\int_{0}^{t_{n+1}}f(t,y(t))(t_{n+1}-t)^{\alpha-1}dt-\int_{0}^{t_n} f(t,y(t))(t_n-t)^{\alpha-1} dt\right]\nonumber\\
&=& \frac{1-\alpha}{ABC(\alpha)}(f_n-f_{n-1})+\frac{\alpha}{ABC(\alpha)\Gamma(\alpha)}\left\{\left[\int_{0}^{t_{n+1}\frac{t-t_{n-1}}{t_n-t_{n-1}}}f(t_n,y_n)\right.\right.\nonumber\\
&&\left. +\frac{t-t_n}{t_{n-1}-t_n}f(t_{n-1},y_{n-1}) +\frac{f^{n+1}(t)}{(n+1)!}\prod_{i=0}^{n}(t-t_i)\right](t_{n+1}-t)^{\alpha-1}dt\nonumber \\
&&-\int_{0}^{t_n}\left[\frac{t-t_{n-1}}{t_n-t_{n-1}}f(t_n,y_n)+ \frac{t-t_{n}}{t_{n-1}-t_{n}}f(t_{n-1},y_{n-1})+\frac{f^{(n)}(t)}{n!}\prod_{i=0}^{n-1}(t-t_i) \right]  \nonumber\\
&&\left.\times (t_n-t)^{\alpha-1}dt \right\}\nonumber\\
&=&L(t,\alpha,n)+\int_{0}^{t_{n+1}}\frac{f^{(n+1)}(t)}{(n+1)!}\prod_{i=0}^{n}(t-t_i)(t_{n+1}-t)^{\alpha-1}dt\nonumber\\
&&- \int_{0}^{t_{n}}\frac{f^{(n)}(t)}{(n)!}\prod_{i=0}^{n}(t-t_i)(t_{n}-t)^{\alpha-1}dt\nonumber
\end{eqnarray}
which implies that
\begin{equation}
y_{n+1}-y_n=L(t,\alpha,n)+R_\alpha(t)
\end{equation}
where
\begin{eqnarray}
L(t,\alpha,n)&=&f(t_n,y_n)\left\{\frac{1-\alpha}{ABC(\alpha)}+\frac{\alpha}{AB(\alpha)h}\left[\frac{2h t^\alpha_{n+1}}{\alpha}-\frac{t_{n+1}^{\alpha+1}}{\alpha+1}\right]-\frac{\alpha}{h\Gamma(\alpha)ABC(\alpha)} \left[\frac{h t^\alpha_{n}}{\alpha}-\frac{t_{n}^{\alpha}}{\alpha+1}\right]\right\}\nonumber\\
&&+f(t_{n-1},y_{n-1})\left\{\frac{\alpha-}{ABC(\alpha)}-\frac{\alpha}{ABC(\alpha)\Gamma(\alpha)h}\left(\frac{h t^\alpha_{n+1}}{\alpha}-\frac{t^{\alpha+1}_{n+1}}{\alpha+1}+\frac{ t^{\alpha+1}_{n}}{h\Gamma(\alpha)ABC(\alpha)} \right) \right\}\nonumber
\end{eqnarray}
and
\begin{equation*}
R_\alpha(t)=\int_{0}^{t_{n+1}}\frac{f^{n+1}(t)}{(n+1)!}\prod_{i=0}^{n}(t-t_i)(t_{n+1}-t)^{\alpha-1}dt-  \int_{0}^{t_{n}}\frac{f^{n}(t)}{(n)!}\prod_{i=0}^{n-1}(t-t_i)(t_{n}-t)^{\alpha-1}dt.
\end{equation*}
We require to show that
\begin{eqnarray}
\|R_\alpha(t)\|_\infty&=&\left\|\int_{0}^{t_{n+1}}\frac{f^{n+1}(t)}{(n+1)!}\prod_{i=0}^{n}(t-t_i)(t_{n+1}-t)^{\alpha-1}dt-\int_{0}^{t_{n}}\frac{f^{n}(t)}{(n)!}\prod_{i=0}^{n-1}(t-t_i)(t_{n}-t)^{\alpha-1}dt \right\|_\infty\nonumber\\
&<&\left\|\int_{0}^{t_{n+1}}\frac{f^{n+1}(t)}{(n+1)!}\prod_{i=0}^{n}(t-t_i)(t_{n+1}-t)^{\alpha-1}dt\right\|_\infty+\left\|\int_{0}^{t_{n}}\frac{f^{n}(t)}{(n)!}\prod_{i=0}^{n-1}(t-t_i)(t_{n}-t)^{\alpha-1}dt \right\|_\infty\nonumber\\
&<&\max_{t\in[0,t_{n+1}]}\frac{\left|f^{(n+1)}(t)\right|}{(n+1)!}\left\|\prod_{i=0}^{n}(t-t_i) \right\|_\infty\frac{t^\alpha_{n+1}}{\alpha} + \max_{t\in[0,t_{n+1}]}\frac{\left|f^{(n)}(t)\right|}{(n)!}\left\|\prod_{i=0}^{n-1}(t-t_i) \right\|_\infty\frac{t^\alpha_{n}}{\alpha}\nonumber\\
&<&\sup_{t\in[0,t_{n+1}]}\left\{\max_{t\in[0,t_{n+1}]}\frac{\left|f^{(n+1)}(t)\right|}{(n+1)!}, \max_{t\in[0,t_{n+1}]}\frac{\left|f^{(n)}(t)\right|}{(n)!}  \right\}\left(n!\frac{h^{n+1}}{4\alpha}t^\alpha_{n+1} +(n-1)! \frac{h^{n}}{4\alpha}t^\alpha_{n} \right).
\end{eqnarray}
This ends the proof.
\end{proof}

\begin{Theorem}{\bf (Stability condition)}
	If $f$ satisfies a Lipschitz condition, then the required stability condition for Adams-Bashforth method when applied to the Atangana-Baleanu fractional derivative in Caputo sense is achieved if
	$$\left\|f(t_n,y_n)-f(t_{n-1},y_{n-1})\right\|_\infty\rightarrow 0$$ as $n\rightarrow \infty$.
\end{Theorem}

\begin{proof}
\begin{eqnarray}
\left\|y_{n+1}-y_n\right\|_\infty&=&\left\|\frac{1-\alpha}{ABC(\alpha)}\left[f(t_n,y_n)-f(t_{n-1},y_{n-1})\right]+\frac{\alpha}{ABC(\alpha)\Gamma(\alpha)}\right.\nonumber\\
&&\left.\times\left[\int_{0}^{t_{n+1}}f(t,y(t))(t_{n+1}-t)^{\alpha-1}dt-\int_{0}^{t_n} f(t,y(t))(t_n-t)^{\alpha-1} dt\right]\right\|_\infty\nonumber\\
&<& \left\|\frac{1-\alpha}{ABC(\alpha)}\left[f(t_n,y_n)-f(t_{n-1},y_{n-1})\right]\right\|_\infty+\frac{\alpha}{ABC(\alpha)\Gamma(\alpha)}\times\nonumber\\
&&\left\|\left[\int_{0}^{t_{n+1}}f(t,y(t))(t_{n+1}-t)^{\alpha-1}dt-\int_{0}^{t_n} f(t,y(t))(t_n-t)^{\alpha-1} dt\right]\right\|_\infty\nonumber\\
&<&\frac{1-\alpha}{ABC(\alpha)}\left\|f(t_n,y_n)-f(t_{n-1},y_{n-1}) \right\|_\infty\nonumber\\
&& +  \frac{\alpha}{ABC(\alpha)\Gamma(\alpha)} \left\|\int_{0}^{t_{n+1}}f(t,y(t))(t_{n+1}-t)^{\alpha-1}dt \right\|_\infty\nonumber\\
&& +  \frac{\alpha}{ABC(\alpha)\Gamma(\alpha)} \left\|\int_{0}^{t_{n}}f(t,y(t))(t_{n}-t)^{\alpha-1}dt \right\|_\infty.\nonumber
\end{eqnarray}
Thus
\begin{eqnarray}
\left\|y_{n+1}-y_n\right\|_\infty&<&\frac{1-\alpha}{ABC(\alpha)}\left\|f(t_n,y_n)-f(t_{n-1},y_{n-1}) \right\|_\infty\nonumber\\
&&+ \frac{\alpha}{ABC(\alpha)\Gamma(\alpha)}\left\|\int_{0}^{t_{n+1}} (t_{n+1}-t)^{\alpha-1}\sum_{i=0}^{n}\prod_{0\le i\le n} \frac{(t-t_i)}{(-1)^ih}f(t_i,y_i)dt \right\|_\infty\nonumber\\
&&+ \frac{\alpha}{ABC(\alpha)\Gamma(\alpha)}\left\|\int_{0}^{t_{n}} (t_{n}-t)^{\alpha-1}\sum_{i=0}^{n-1}\prod_{0\le i\le n-1} f(t_i,y_i)dt \right\|_\infty\nonumber\\
&<&\frac{1-\alpha}{ABC(\alpha)}\left\|f(t_n,y_n)-f(t_{n-1},y_{n-1}) \right\|_\infty + \left\|P_n^\alpha(t) \right\|_\infty+\left\|R_n^\alpha(t) \right\|_\infty
\end{eqnarray}
where
\begin{eqnarray}
\left\|P_n^\alpha(t) \right\|_\infty&=&\left\|\int_{0}^{t_{n+1}}(t_{n+1}-t)^{\alpha-1}\sum_{i=0}^{n}\frac{(t-t_i)}{(-1)^ih}f(t_i,y_i)dt \right\|_\infty\nonumber\\
&\le& \sum_{i=0}^{n}\frac{\|f(t_i,y_i) \|_\infty}{h} \frac{t^\alpha_{n+1}}{\alpha}\prod_{i=0}^{n}|t-t_i|\nonumber\\
&\le& \sum_{i=0}^{n}\frac{\|f(t_i,y_i) \|_\infty}{h} \frac{t^\alpha_{n+1}}{\alpha}\frac{n! h^n}{4}
\end{eqnarray}
and
\begin{equation}
\left\|R_n^\alpha(t) \right\|_\infty \le \sum_{i=0}^{n}\frac{\|f(t_i,y_i) \|_\infty}{h} \frac{t^\alpha_{n}}{\alpha}(n-1)!\frac{h^{n-1}}{4}.
\end{equation}
Hence,
\begin{eqnarray}
\left\|y_{n+1}-y_n\right\|_\infty&<&\frac{1-\alpha}{ABC(\alpha)}\left\|f(t_n,y_n)-f(t_{n-1},y_{n-1}) \right\|_\infty\nonumber\\
&&+\sum_{i=0}^{n}\frac{\|f(t_i,y_i) \|_\infty}{4\alpha}t^{\alpha}_{n+1}h^{n-1}n!+ \sum_{i=0}^{n-1}\frac{\|f(t_{i-1},y_{i-1}) \|_\infty}{4\alpha}t^{\alpha}_{n}h^{n-3}(n-1)!\nonumber\\
&<&\frac{Mn!h^n}{4\alpha}\left\{\frac{t_{n+1}^\alpha(n+1)}{h}+\frac{t_n^\alpha}{h^2} \right\} +  \frac{1-\alpha}{ABC(\alpha)}\left\|f(t_n,y_n)-f(t_{n-1},y_{n-1}) \right\|_\infty.\nonumber
\end{eqnarray}
where $\displaystyle M=\max_{t\in[0,t_{n+1}]}|f(t,y(t))|$, for $n\rightarrow \infty$ as $\left\|f(t_n,y_n)-f(t_{n-1},y_{n-1}) \right\|\rightarrow 0$ and $\frac{Mn!h^n}{4\alpha}\rightarrow 0$ as $h\rightarrow 0$. The proof is completed.
\end{proof}
%%===================================================================

%%===================================================================
\section{Numerical experiments}
In this section, we experiment the performance of the derived  fractional Adams-Bashforth schemes for the Caputo, Caputo-Fabrizio and Atangana-Baleanu derivatives.

We consider the fractional Fisher's equation
\begin{equation}\label{e1}
D_t^\alpha u(x,t)=\delta\frac{\partial^2u(x,t)}{\partial x^2}+u(x,t)(1-u(x,t))+f(x,t)
\end{equation}
subject to the initial and Neumann boundary conditions
\begin{eqnarray}\label{e2}
u(x,0)&=&\cos(5\pi x)\nonumber\\
u_x(0,t)&=&t^3,\;\;\;\; u_x(1,t)=2t^4+t^3,\;\;t\ge 0,\\
u(x,t)&=&(t^\tau+1)\cos(5\pi x)+t^4 x^2+t^3x,\nonumber
\end{eqnarray}
and
\begin{eqnarray}\label{e3}
f(x,t)&=&\cos(5\pi x)\left[\frac{\Gamma(\tau+1)}{\Gamma(\tau+1-\alpha)}t^{\tau-\alpha}+25\pi^2(t^\tau+1)-(t^\tau+1) +(t^{\tau+1})^2\cos(5\pi x) \nonumber\right.\\
&&\left. +2t^3(t^\tau+1)x+2t^4(t^{\tau+1})x^2\right]+\frac{\Gamma(5)}{\Gamma(5-\alpha)}t^{4-\alpha}x^2 + \frac{\Gamma(4)}{\Gamma(4-\alpha)}t^{3-\alpha}x\nonumber\\
&&-2t^4-x^2t^4-t^3x+t^6x^2+t^8x^4+2t^7x^3,
\end{eqnarray}
where the fractional derivative in (\ref{e1}) is taken to be the Caputo, Caputo-Fbrizio and Atangana-Baleanu derivatives.

Next, we discretize the space and time derivatives. To approximate the space derivative, we consider a uniform mesh on the interval $[0, L]$, defined by the grid-points $x_i=i\Delta x$ for $i=0,1,2,\ldots,N$, where $\Delta x=\frac{L}{N}$, and at every $x=x_i,\;i=1,2,\ldots,N-1$. For the second order partial derivatives, we consider the the second-order finite difference
\begin{equation}
\frac{\partial^2u(x_i,t)}{\partial x^2}=\frac{u(x_{i+1},t)-2u(x_i,t)+u(x_{i-1},t)}{(\Delta x)^2}-\frac{(\Delta x)^2}{12} \frac{\partial^4u}{\partial x^4}(E_i,t), \;\text{for}\;E_i\in(x_{i-1},x_{i+1}).
\end{equation}

The results in Tables \ref{table1} and \ref{table2} show the performance of the fractional Adams-Bashforth schemes in conjunction with various derivatives as given in the table captions. We list the maximum errors computed as $\|E\|=\max|U_e-U_c|$, where $U_e$ and $U_c$ are the exact and computed results.

\begin{table}[ht]
	\caption{Maximum error results for equation (\ref{e1}) with the Caputo, Caputo-Fabrizio and Atangana-Baleanu derivatives at $\delta=10, \tau=1, \alpha=0.35, t=0.5, L=1$, simulation runs for $N=100$.} \centering
	\begin{tabular}{c c c c c}
		\hline
		$\Delta t$ & $\Delta x$ & Caputo & Caputo-Fabrizio & Atangana-Baleanu  \\ 	\hline
		0.25 &0.5&6.6656e-06 & 4.6187e-06 & 1.4782e-06   \\
		0.0625 &0.25 &1.0653e-06 & 7.1804e-07 & 2.2827e-07   \\
		0.015625 & 0.1250 &3.3161e-07 & 2.1293e-07 & 6.6861e-08     \\
		0.00390625 &0.0625& 1.3995e-07 & 8.3324e-08 & 2.5625e-08   \\[1ex]
		\hline
	\end{tabular}
	\label{table1}
\end{table}

\begin{table}[ht]
	\caption{Maximum error and timing results of (\ref{e1}) obtained at some different values of $\alpha$ with the  Caputo, Caputo-Fabrizio and Atangana-Baleanu derivatives at $\delta=1.0, \tau=1, \Delta t=0.05, \Delta x=0.25, t=1.0, L=1$, simulation runs for $N=100$.  } \centering
	\begin{tabular}{c c c c c c c}
	\hline
	$\alpha$ & Caputo & CPU & Caputo-Fabrizio & CPU & Atangana-Baleanu & CPU \\ [0.5ex]\hline
	0.21 &6.7827e-06 & 0.18 & 4.3656e-07 & 0.17 & 9.8489e-08 & 0.18 \\
	0.43 & 1.0663e-05 & 0.18 & 1.0118e-06 & 0.18 & 2.2731e-07 & 0.18 \\
	0.65 & 7.8794e-06 & 0.18 & 1.0197e-06 & 0.18 & 2.2784e-07 & 0.18  \\
	0.89 & 2.9779e-06 & 0.18 & 4.1988e-07 & 0.17 & 8.9765e-08& 0.17 \\[1ex]
		\hline
	\end{tabular}
	\label{table2}
\end{table}

\begin{figure}[!h]
	\centering
	\begin{tabular}{c}
		\begin{minipage}{400pt}
			\includegraphics[width=400pt]{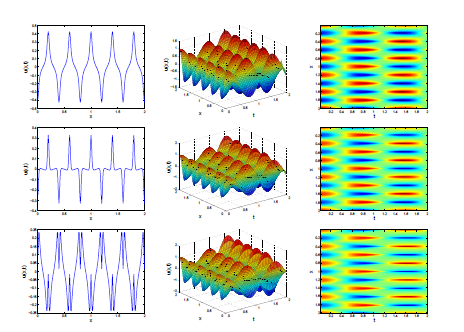}
		\end{minipage}
		\end{tabular}
	\caption{Spatial evolution of fractional Fisher's equation (\ref{e1}) obtained at $\delta=0.1, \alpha=0.35, \tau=1.0, t=2$ for different derivatives. The upper-, middle- and lower-rows correspond to the Caputo, Caputo-Fabrizio and the Atangana-Baleanu derivatives respectively.}\label{Fig1}
\end{figure}

\begin{figure}[!h]
 	\centering
 	\begin{tabular}{c}
 		\begin{minipage}{400pt}
 			\includegraphics[width=400pt]{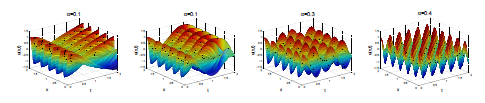}
 		\end{minipage}
 	\end{tabular}
 	\caption{One-dimensional results showing the distribution of $u(x,t)$ at different instances of $\alpha=0.1(0.1)0.4$ for $t=1$. Other parameters are given in Figure \ref{Fig1}. }\label{Fig2}
 \end{figure}

 \begin{figure}[!h]
 	\centering
 	\begin{tabular}{c}
 		\begin{minipage}{400pt}
 			\includegraphics[width=400pt]{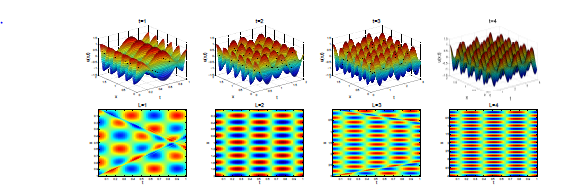}
 		\end{minipage}
 	\end{tabular}
 	\caption{Results showing the effects of time $t$ (upper-row) and spatial domain $L$ (lower-row) Other parameters are given in Figure \ref{Fig1}. }\label{Fig3}
 \end{figure}

For the second experiment, we simulate equations (\ref{e1}-\ref{e3}) on a spatial domain $\Omega\in[0, L]$, for $L>0$ with $\Delta t=0.25$, $\Delta x=0.5$ and $\tau=0.9$ with $N=200$. In Figure \ref{Fig1}, we demonstrate the behaviours of the fractional derivatives for $\alpha=0.35$. The upper, middle and lower rows correspond to the Caputo, Caputo-Fabrizio and Atangana-Baleanu fractional derivatives. Merely looking at the figures, one may conclude that the three derivatives yielded a  similar results. But a keen look will observe that they are not similar, column-1 of Figure \ref{Fig1} is a proof to this assertion.

The effects of fraction index $\alpha$ is observed in Figure \ref{Fig2} with the Caputo derivative. Likewise in Figure \ref{Fig3} (upper-row), we examined the distribution of $u(x,t)$ at some instances of final simulation time $t$. As the time is increasing so also the number of oscillations. In the lower panels, we observed the effects of increasing the domain size $L$, different patterns are obtained. It should be noted that other structures such as pure-spots, stripe and spatiotemporal patterns apart from what is reported in this work can be obtained, depending on how parameter values are chosen.

%%=================================================================
\section{Conclusion}
This paper has proposed the correct version of the fractional Adams-Bashforth methods which take into account the nonlinearity of the kernels including the power law for the Riemann-Liouville type, the exponential decay law for the Caputo-Fabrizio case and the Mittag-Leffler law for the Atangana-Baleanu scenario.
The stability, as well as convergence results for each of the derivatives, are clearly presented. We compute the maximum norm error to check the performances of these schemes via the fractional Fisher's equation at some instances of fractional-order $\alpha$. Formulation of space fractional Adams-Bashforth scheme with the Caputo, Caputo-Fabrizio and Atangana-Baleanu Riemann-Liouville derivatives, as well as their applications to real life problems are left for future research.
%%=================================================================
%\section*{Acknowledgment}
%The authors are grateful to all of the anonymous reviewers for their valuable suggestions
%The author is grateful to all of the anonymous reviewers for their valuable suggestions

%%===================================================================
%\begin{figure}[!h]
%	\centering
%	\begin{tabular}{cc}
%		\begin{minipage}{250pt}
%			\includegraphics[width=250pt]{ADV1}
%		\end{minipage}
%		\begin{minipage}{250pt}
%			\includegraphics[width=250pt]{ADV2}
%		\end{minipage}
%	\end{tabular}
%	\caption{Chaotic attractors and time series plots for R\"ossler system (\ref{Rose}) showing spiral-type chaos at two instances of $c$, the initial conditions are $u(0)=v(0)$=w(0)=1. The upper- and lower-rows correspond to $c=3.1$ and $c=6.2$ respectively. Simulation runs for $t=100$.  }\label{Fig1}
%\end{figure}
%%===================================================================

\end{document}